\documentclass{birkjour}
 \newtheorem{thm}{Theorem}[section]
 \newtheorem{cor}[thm]{Corollary}
 \newtheorem{lem}[thm]{Lemma}
 
 \theoremstyle{definition}
 
 \theoremstyle{remark}
 \newtheorem{rem}[thm]{Remark}
 
 \numberwithin{equation}{section}

\begin{document}

%
%
%
%
%
%
%
%
%

\title[ON A GENERALIZATION FOR QUATERNION SEQUENCES]
 {ON A GENERALIZATION FOR \\QUATERNION SEQUENCES}

\author[Serp\.{i}l HALICI]{Serp\.{i}l HALICI}

\address{%
Pamukkale University,\\
Faculty of Arts and Sciences,\\
Department of Mathematics,\\
Denizli/TURKEY}

\email{shalici@pau.edu.tr}



\date{2016}

\begin{abstract}
In this study, we introduce a new classes of quaternion numbers. We show that this new classes quaternion numbers include all of quaternion numbers such as Fibonacci, Lucas, Pell, Jacobsthal, Pell-Lucas, Jacobsthal-Lucas quaternions have been studied by many authors. Moreover, for this newly defined quaternion numbers we give the generating function, norm value, Cassini identity, summation formula and their some properties.
\end{abstract}

\maketitle
\section{Introduction}
Quaternions play an important role in quantum physics, computer science and in many areas of mathematics. Also, quaternionic numbers are useful tools to the representations and generalizations of quantities in the high-dimensional physics theory. Classically quaternions are represented in the form of hyper-complex numbers with three imaginary components;
\begin{equation} q=(a_0, a_1, a_2, a_3)=a_0+a_1i+a_2j+a_3k \end{equation}
where $i,j$ and $k$ mutually unit bivectors and $a_0, a_1, a_2, a_3$ are real numbers. The vectors  $i, j, k $ obey the famous multiplication rules; $i^2=j^2=k^2=ijk=-1$, discovered by Hamilton in $1843$. The conjugate of any quaternion $q$ is given by negating the three imaginary components[13]:
\begin{equation} \overline{q}=(a_0,-a_1,-a_2,-a_3)=a_0-a_1i-a_2j-a_3k. \end{equation}
For any quaternions $p$ and $q$, $\overline{q}\overline{p}=\overline{pq}$ can be written. The norm value of a quaternion $q$ is given by the following formula[13]:
\begin{equation} ||q||=q\overline{q}=\overline{q}q. \end{equation}
The modulus of quaternion $q$ is defined as the square root of its norm that is $|q|=\sqrt{||q||}.$
Every non zero quaternion $q$ has a multiplicative inverse that can be stated as follows:
$$q^{-1}=\frac{\overline{q}}{||q||}.$$
Since there is a multiplicative inverse for any non zero quaternion $q$ and for any two quaternions $p$ and $q$ the property $||pq||=||p||||q||$ is satisfied in the quaternion algebra and $\mathbb{H}$ is a normed division algebra. In addition to this we know that $pq\neq qp$. Particulary, more detailed and useful knowledge can be found in [12, 13]. \\\\
Different types of quaternions have been studied by many mathematicians[3, 4, 5, 6, 14, 16, 17]. One of the earliest studies on this subject belong to Horadam[8] where it was defined Fibonacci quaternions. In his study author gave some quaternions recurrence relations. Swamy[15] defined the generalized Fibonacci sequence. In [5] Halici gave the Binet formula, generating function and some identities for Fibonacci quaternions. In[14] Polatli and Kesim investigated the quaternions with generalized Fibonacci and Lucas number components. In [3] Bolat and Ipek studied Pell and Pell-Lucas quaternions. In [16, 17] Liano and Wloch introduced Pell and Jacobsthal quaternions, respectively, but it should be noted that they did not give the Binet formulas of these quaternions. In [2] Catarino studied modified Pell quaternions and gave its norm value, the generating function, Binet formula, Cassini identity.\\\\
Inspired by these works, in this paper we define a new generalization of sequences of quaternions and we call them as Horadam quaternions. It should be noted that this new sequence generalize the Fibonacci, Lucas, Pell, Pell-Lucas, Jacobsthal and Jacobsthal-Lucas quaternions which are studied separately by some authors[5, 13, 3, 16].

\section{\textbf{HORADAM SEQUENCE}}
We start by recalling some fundamental properties of Horadam numbers. Horadam defined the following sequence[7];
\begin{equation} W_n = W_n (a, b; p, q)=pW_{n-1}+qW_{n-2} ;\,\  n\geq 2,\,\ W_0 = a, \,\ W_1 = b \end{equation}
In here $W_n$ is called as the $nth$ Horadam number which is a Fibonacci type number defined recursively by the second order linear recurrence relation. Then its the characteristic equation is $x^2-px-q=0$
and hence the roots of it are $\alpha=\frac{p+\sqrt{p^2+4q}}{2}$ and $\beta=\frac{p-\sqrt{p^2+4q}}{2}$. It is well known that for the Horadam numbers Binet formula is

\begin{equation}W_n=\frac{A\alpha^n-B\beta^n}{\alpha-\beta}; A=b-a\beta, B=b-a\alpha. \end{equation}
In general the sequence $(W_{n})$ is known as Horadam sequence. In fact Horadam sequence gives the well-known special sequences such as Fibonacci, Pell, Jacobsthal, Pell-Lucas, Jacobsthal-Lucas, Tagiuri, Fermat, Fermat-Lucas. For example, some of these sequences are
\begin{equation} (W_n)={W_n(0,1;1,1)}_0^\infty, \,\  (W_n)={W_n(2,1;1,1)}_0^\infty \end{equation}
which are Fibonacci and Lucas sequences, respectively.\\\\
As well as the formula $(2.2)$, for Horadam sequence another Binet formula can be given as in the following Lemma.

\begin{lem}
For Horadam sequence, we have
\begin{equation} W_n=bT_n+aqT_{n-1}\end{equation}
where $a, b, p $ and $q$ are initial values and
\begin{equation} T_n=\frac{1}{\sqrt{p^2+4q}} \left \{  \Big( \frac{p+\sqrt{p^2+4q}}{2} \Big)^n- \Big( \frac{p-\sqrt{p^2+4q}}{2}\Big)^n \right \} \end{equation}
\end{lem}
\begin{proof}
From the recurrence relation and Binet formula of Horadam sequence the proof is easily verified.
\end{proof}
It should be noted that choosing $(0, 1; p, q)$ and $(2, p; p, q)$ instead of the values $a, b, p $ and $q$ in the formula $(2.4)$, respectively, generalized Fibonacci and Lucas sequences are obtained. That is
\begin{equation} W_n(0,1;p,q)=T_n=\Big( \frac{\alpha^n-\beta^n}{\alpha-\beta} \Big) \end{equation} and
\begin{equation} W_n (2,p;p,q)=T_n=\alpha^n+\beta^n .\end{equation}
Specially, if we write $p= q= 1$ in the equation $(2.5)$ and using the formula $(2.4)$, then we have the Binet formula Fibonacci sequence's;

\begin{equation} W_n(0,1;p,q)=T_n= \frac{1}{\sqrt{5}} \left \{  \Big( \frac{1+\sqrt{5}}{2} \Big)^n- \Big( \frac{1-\sqrt{5}}{2}\Big)^n \right \}. \end{equation}

\section{\textbf{HORADAM QUATERNIONS}}

In this section we define Horadam quaternions which generalize all known quaternions. Also, we give some properties related to these quaternions such as Binet formula, the generating function, Cassini formula and some summation formulas. \\\\
Firstly, for $n\geq 3$ and integers $p, q$ we point out the following definition by Swamy[15].
\begin{equation}
H_n=H_{n-1}+H_{n-2};\,\  H_1=p, \,\  H_2= p+q .\end{equation}
In fact the formula $(3.1)$ is a generalization for Fibonacci numbers.The author gave also the following recurrence relation for this sequence:
\begin{equation}
H_{n+1}=qF_{n}+pF_{n+1}.\end{equation}
And then he provided a new generalization for  Fibonacci quaternions; it is this:
\begin{equation} P_n = H_n + i H_{n+1} + j H_{n+2} + k H_{n+3}. \end{equation}
Moreover, he obtained some useful equations involving these quaternions as follows.
\begin{equation} P_n \overline{P_n} = \overline{P_n} P_n = 3(2pq-q^2)F_{2n+2} + (p^2 + q^2) F_{2n+3}, \end{equation}
\begin{equation} P_n^2 + P_{n-1}^2=2(H_n P_n+ H_{n-1}P_{n-1})-(P_n \overline{P_n} + P_{n-1} \overline{P_{n-1}}), \end{equation}
\begin{equation} P_n\overline{P_n} + P_{n-1}\overline{P_{n-1}}=3(p^2 L_{2n+2}+2pq L_{2n+1} + q^2 L_{2n}). \end{equation}
Now, in a similar to Horadam's definition for the integers greater than zero and zero we will define a new quaternion as follows.
\begin{equation} Q_{w,n+2}=W_{n+2}1+W_{n+3} i +W_{n+4} j +W_{n+5} k \end{equation}
where $W_n$ is the $nth$ generalized Horadam number. Due to its components this new quaternion can be called Horadam quaternion. After some necessary calculations we obtain the following recurrence relation;
\begin{equation} Q_{w,n+2}=pQ_{w,n+1}+qQ_{w,n}. \end{equation}
We will show that this new quaternionic sequence generalize all the type of quaternion sequences.
Let us calculate the initial values of the sequence in the equation $(3.7)$;
\begin{equation} Q_{w,0}=(a, b, pb+qa, p^2b+pqa+qb), \end{equation}
\begin{equation} Q_{w,1}=(b, pb+qa, p^2b+pqa+qb, p^3b+p^2qa+2pqb+q^2a). \end{equation}\\\\
It should be noted that taking $0,1$ $p$ and $1$ instead of $a,b,p,$ and $q$ in the equations $(3.9)$ and $(3.10)$
\begin{equation} Q_{w,0}=(0, 1, p, p^2+1) \,\ \mbox{and} \,\ Q_{w,1}=(1,p,p^2+1,p^3+2p) \end{equation} are obtained, respectively. $Q_{w,0}$ and $Q_{w,1}$ are the initial values of generalization Fibonacci and Lucas quaternions which are given by Polatli and Kesim in [14]. Since the quadratic equation of the recurrence relation $(3.8)$ is $\lambda^2-p\lambda-q=0,$ the roots of this equation are
\begin{equation} \lambda_1=\alpha=\frac{p+\sqrt{p^2+4p}}{2}, \,\  \lambda_2=\beta=\frac{p-\sqrt{p^2+4q}}{2} \end{equation}
In [5] we provided that Binet formula for Fibonacci quaternion is
\begin{equation} Q_n=\frac{1}{\sqrt{5}}(\underline{\alpha} \alpha^n - \underline{\beta} \beta^n) \end{equation}
where $\alpha$ and $\beta$ are the roots of Fibonacci sequence and
\begin{equation} \underline{\alpha}=1+i\alpha+j\alpha^2+k\alpha^3 , \,\ \underline{\beta}=1+i\beta+j \beta^2+k\beta^3 \end{equation}
In the next theorem we will give Binet formula for the Horadam quaternions.
\begin{thm} The Binet formula of Horadam quaternions is
\begin{equation} Q_{w,n}=\frac{1}{\alpha - \beta} (A\underline{\alpha} \alpha^n - B\underline{\beta}\beta^n) =bT_n+ aqT_{n-1} \end{equation}
where $A=b-a\beta, \,\  B=b-a \alpha$ and $ T_{n} $ is as the equation $(2.5)$.
\end{thm}
\begin{proof}
Using the definition of Horadam quaternions and its Binet formula we write
\scriptsize { $$Q_{w,n}=\frac{A\alpha^n-B\beta^n}{\alpha-\beta}+\frac{A\alpha^{n+1}-B\beta^{n+1}}{\alpha-\beta}i+\frac{A\alpha^{n+2}-B\beta^{n+2}}{\alpha-\beta}j+\frac{A\alpha^{n+3}-B\beta^{n+3}}{\alpha-\beta}k$$}
\normalsize And then by some direct computations the equation $(3.15)$ is obtained.\end{proof}
In the following Remark we drive the Binet formulas for different types of quaternion sequences.

\begin{rem} We can list results of the Theorem $ (3.1) $  as follows:\\\\
i) In the equation $(3.15)$ if we choose $a = 0, b = 1$ and $ p= q =1$, then we get
 \begin{equation} \scriptsize Q_{w,n}=\frac{1}{\alpha - \beta} (\underline{\alpha} \alpha^n - \underline{\beta}\beta^n); \,\ \underline{\alpha}=1+i\alpha+j\alpha^2+k\alpha^3 , \,\  \underline{\beta}=1+i\beta+j \beta^2+k\beta^3
\end{equation}
which is Binet formula for Fibonacci quaternions. The formula $(3.16)$ is given by Halici in [5].\\\\
ii) In the equation $(3.15)$ by taking $a=0, b=1, p=2 $ and $ q=1 $ we have
\begin{equation} Q_{w,n}=\frac{1}{2\sqrt{2}} ( \underline{\alpha}(1+\sqrt{2})^n-\underline{\beta}(1-\sqrt{2})^n ) \end{equation}
which is the Binet formula for Pell quaternions. It should be noted that the Pell quaternions are studied and its Binet formula is given by Bolat and Ipek in [3] in which this formula the same as $(3. 17) $.\\\\
iii) In the equation $(3.15)$ if we take $a=0, b=1, p=1$ and $ q=2$, then we have
\begin{equation} Q_{w,n}=\frac{1}{\alpha-\beta} ( \underline{\alpha}\alpha^n-\underline{\beta}\beta^n)=\frac{1}{3} ( \underline{\alpha}2^n-\underline{\beta}(-1)^n ) \end{equation}
where $\underline{\alpha}=1+2i+4j+8k$ and $\underline{\beta}=1-i+j-k$. The equation $(3.18)$ gives $nth$ Jacobsthal quaternion. We state that this quaternion type is considered by Szynal-Lianna and Wloch in [17], but they did not give its Binet formula.\\\\
iv) In the equation $(3.15)$ if we choose $a=2, b=1, p=q=1$, then we get the Binet formula of Lucas quaternions as follows.
\begin{equation} Q_{w,n}=\underline{\alpha}\alpha^n+\underline{\beta}\beta^n=K_n \end{equation}
Here \scriptsize { $$\underline{\alpha}=1+i\alpha+j\alpha^2+k\alpha^3;\,\  \underline{\beta}=1+i\beta+j\beta^2+k\beta^3 ; \,\  \alpha=\frac{1+\sqrt{5}}{2}, \,\ \beta=\frac{1-\sqrt{5}}{2}.$$ } \\\\
\normalsize v) If we choose $a=2, b=1, p=2, q=1$ in the equation $(3.15)$ then we have the following formula
\begin{equation} Q_{w,n}=\frac{1}{\alpha-\beta}(\underline{\alpha}(1-2\beta)\alpha^n-\underline{\beta}(1-2\alpha)\beta^n).\end{equation}
The formula $(3.20)$ gives the $nth$  Pell-Lucas quaternion.\\\\
vi) If we define a new quaternion type with Jacobsthal-Lucas number components as follows.
\begin{equation} Q_{jn}=j_n+ij_{n+1}+jj_{n+2}+kj_{n+3} , \end{equation} and if we use the equations $  j_{n+2}=j_{n+1}+2j_n; \,\  j_0=2, j_1=1 $ and $j_n=2^n-(-1)^n$, then we obtain
\begin{equation} Q_{w,n}=(1+2i+4j+8k)\alpha^n - (1-i+j-k) \beta^n \end{equation}
The equation $(3.22)$ gives the $nth$ Jacobsthal-Lucas quaternion.
\end{rem}
Generating functions provide a powerful tool for solving linear homogeneous recurrence relations with constant coefficients.These functions are used firstly in order to solve the Fibonacci recurrence relation by A. De Moivre.
More generally, in the following theorem we present the generating function for Horadam quaternions.
\begin{thm}
The generating function of Horadam quaternions is
\begin{equation} g(t)=\frac{Q_{w,0}+(Q_{w,1}-pQ_{w,0})t}{1-pt-qt^2}. \end{equation}
Here $Q_{w,0}$ and $Q_{w,1}$ are the initial values for Horadam quaternions.
\end{thm}
\begin{proof}
Let $g(t)$ be generating function for Horadam quaternions. Then we have
\begin{equation} g(t)=\sum^\infty_{n=0}{Q_{w,n}t^n}. \end{equation}
By multiplying both sides of the equation $(3.24)$ by $pt$ and $qt^2$ we get
$$ptg(t)=\sum^\infty_{n=0}{pQ_{w,n}t^{n+1}} \mbox{ and } qt^2g(t)=\sum^\infty_{n=0}{qQ_{w,n}t^{n+2}}.$$
From the recurrence relation $Q_{w,n+2}=pQ_{w,n+1}+qQ_{w,n}$ we get\\
\begin{equation} g(t)=\frac{Q_{w,0}+(Q_{w,1}-pQ_{w,0})t}{1-pt-qt^2} \end{equation}
is desired.
\end{proof}
In the following Remark we will investigate some special cases of the generating function which is given in $(3.23)$.
\begin{rem} We can list results of the Theorem $ (3.3) $  as follows:\\\\
i) In the equation $(3.25)$  if we write the values $ a=0, b=1, p = q= 1$, then we get
\begin{equation} g(t)=\frac{(0,1,1,2)+(1,0,1,1)t}{1-t-t^2}=\frac{t+i+(1+t)j+(2+t)k}{1-t-t^2}. \end{equation}
The equation $(3.26)$ gives the generating function of Fibonacci quaternions. This function is given by Halici in [5]. \\\\
ii) In the equation $(3.25)$ if we write the values $a=2, b=1, p=q=1$, then we have
\begin{equation}\scriptsize   g(t)=\frac{(2,1,3,5)+(-1,2,2,2)t}{1-t-t^2}=\frac{(2-t)+(1+2t)i+(3+2t)j+(5+2t)k}{1-t-t^2} \end{equation}
which is the generating function for the Lucas quaternions. \\\\
iii) If we write $a=0, b=1, p=2, q=1$ in the equation $(3.25)$, then we have the following function.
\begin{equation} \scriptsize g(t)=\frac{(0,1,2,5)+(1,0,1,2)t}{1-2t-t^2}=\frac{t+i+(2+t)j+(5+2t)k}{1-2t-t^2}. \end{equation}
It should be noted that this type quaternion is studied by Bolat and Ipek in [3], but they did not give the generating function of them. \\\\
iv) In the equation $(3.25)$ if we write the values $a=0, b=1, p=1, q=2$, then we have
\begin{equation} \scriptsize  g(t)=\frac{(0,1,1,3)+(1,0,2,2)t}{1-t-2t^2}=\frac{t+i+(1+2t)j+(3+2t)k}{1-t-2t^2} \end{equation}
which is a generating function for the Jacobsthal quaternions.\\\\
v) If we write $a=2, b=1, p=2, q=1$ in the equation $(3.25)$, then we have
\begin{equation} \scriptsize g(t)=\frac{(2,1,4,9)+(-3,2,1,4)t}{1-2t-t^2}=\frac{(2-3t)+(1+2t)i+(4+t)j+(9+4t)k}{1-2t-t^2} \end{equation}
which is the generating function for the Pell-Lucas quaternions.\\\\
vi) If we write $a=2, b=1, p=1, q=2$ in the equation $(3.25)$, then we have
\begin{equation} \scriptsize  g(t)=\frac{(2,1,5,7)+(-1,4,2,10)t}{1-t-2t^2}=\frac{(2-t)+(1+4t)i+(5+2t)j+(7+10t)k}{1-t-2t^2} \end{equation}
which is the generating function for the Jacobsthal-Lucas quaternions.
\end{rem}
It is well known that the Cassini formula is one of the oldest identities involving the Fibonacci type numbers. In the following theorem we will give the Cassini identity related with the Horadam quaternions.
\begin{thm}
The Cassini formula for Horadam quaternions is follows.
\begin{equation} Q_{w,n-1} Q_{w,n+1} - Q_{w,n}^2=\frac{AB\alpha^{n-1}\beta^{n-1}}{\alpha-\beta}(\beta \underline{\alpha}\underline{\beta}-\alpha \underline{\beta}\underline{\alpha}). \end{equation}
\end{thm}
\begin{proof}
Using the Binet formula, it can be shown that
\scriptsize $$\scriptsize Q_{w,n-1} Q_{w,n+1} - Q_{w,n}^2 =\frac{1}{(\alpha-\beta)^2} \big\{ \big( A\underline{\alpha} \alpha^{n-1} - B \underline{\beta} \beta^{n-1} \big) \big(A\underline{\alpha} \alpha^{n+1} - B \underline{\beta} \beta^{n+1} \big)-\big(A\underline{\alpha} \alpha^{n} - B \underline{\beta} \beta^{n} \big)^2 \big\} $$

\scriptsize $$ Q_{w,n-1} Q_{w,n+1} - Q_{w,n}^2 =\frac{1}{(\alpha-\beta)^2}\big(-AB\underline{\alpha}\underline{\beta}\alpha^{n-1}\beta^{n+1}-AB\underline{\beta}\underline{\alpha}\alpha^{n+1}\beta^{n-1} + AB \underline{\alpha}\underline{\beta}\alpha^n\beta^n+AB\underline{\beta}\underline{\alpha}\alpha^n\beta^n\big)$$
\normalsize Thus, we have the following equation.
$$ Q_{w,n-1} Q_{w,n+1} - Q_{w,n}^2 =\frac{AB\alpha^{n-1}\beta^{n-1}}{\alpha-\beta}(\beta \underline{\alpha}\underline{\beta}-\alpha \underline{\beta}\underline{\alpha}) $$
\end{proof}
Consequently, we can give the following Remark which gives some special cases of the Horadam quaternions's Cassini formula.
\begin{rem} We can interpret results of the Theorem $ (3.5) $  as follows:\\\\
In fact in the equation $(3.32)$ writing the values $A, B $ and $ \alpha, \beta $  for Fibonacci quaternions one can get the following equation.
\begin{equation}\scriptsize  Q_{w,n-1} Q_{w,n+1} - Q_{w,n}^2=\frac{AB\alpha^{n-1}\beta^{n-1}}{\alpha-\beta}(\beta \underline{\alpha}\underline{\beta}-\alpha \underline{\beta}\underline{\alpha})=(-1)^n (2+2i+4j+3k). \end{equation}
From [5] we know that
\begin{equation}\scriptsize  Q_{w,n-1} Q_{w,n+1} - Q_{w,n}^2=(-1)^n (2Q_1 -3 k) \end{equation}
where $Q_n$ is $nth$ Fibonacci quaternion. After the needed calculations one can see that the above last two equations are the same. Thus,
$$ Q_{w,n-1} Q_{w,n+1} - Q_{w,n}^2=(-1)^n (2Q_1 -3 k)= (-1)^n (2+2i+4j+3k) $$ is written.
 \\\\
Similarly, in the formula $(3.32)$ writing the needed values for Pell quaternions
\begin{equation} Q_{w,n-1} Q_{w,n+1} - Q_{w,n}^2=\frac{(-1)^{n+1}}{4}(\underline{\alpha}\underline{\beta}(\alpha^2+2)-\underline{\beta}\underline{\alpha}\beta^2)\end{equation}
is obtained. This last formula is the Cassini formula for Pell quaternions which is given by Bolat and Ipek in [3]. \end{rem}
A summation formula for Horadam quaternions $ Q_{w,n}$ can be given as follows.
\begin{thm}
A summation formula for Horadam quaternions is
\begin{equation} \sum_{k=0}^n{Q_{w,k}}=\frac{1}{\alpha-\beta}\big( \frac{B\underline{\beta}\beta^{n+1}}{1-\beta}-\frac{A\underline{\alpha}\alpha^{n+1}}{1-\alpha} \big) + K\end{equation}
where $A=b-a\beta, B=b-a\alpha$ and
\scriptsize $$  K=\frac{(a+b-ap)+i(b+aq)+j(bp+aq+bq)+k[b(p^2+q)+(a+b)pq+aq^2]}{1-p-q}.$$
\end{thm}
\begin{proof}
Using the Binet formula of Horadam quaternions we can state
$$\sum_{k=0}^n{Q_{w,k}}=\sum_{k=0}^n \frac{1}{\alpha-\beta} \big(A\underline{\alpha}\alpha^n - B\underline{\beta} \beta^n \big)=\frac{A\underline{\alpha}}{\alpha-\beta} \sum_{k=0}^n{\alpha^n}-\frac{B\underline{\beta}}{\alpha-\beta} \sum_{k=0}^n{\beta^n}.$$
By the aid of the values $A=b-a\beta, \,\ B=b-a\alpha$ and the definition of geometric series we get
$$\sum_{k=0}^n{Q_{w,k}}=\frac{(b-a\beta) \underline{\alpha}(1-\alpha^{n+1})}{(\alpha-\beta)(1-\alpha)}-\frac{(b-a\alpha) \underline{\beta}(1-\beta^{n+1})}{(\alpha-\beta)(1-\beta)}$$
From some straightforward computations it follows that
$$\sum_{k=0}^n{Q_{w,k}}=\frac{1}{\alpha-\beta}\big( \frac{B\underline{\beta}\beta^{n+1}}{1-\beta}-\frac{A\underline{\alpha}\alpha^{n+1}}{1-\alpha} \big) + K$$
which is desired.
\end{proof}
We will show that the sum formula in the equation $(3.36)$ is a generalization for all the other quaternion types. For this purpose we give the next Corollary.

\begin{cor} In the equation $(3.36)$ the sum formula is a generalization for all quaternion types.

\end{cor}
\begin{proof}
Firstly, we write the values $a=0, b=p=q=1$ in the equation $(3.36)$. Then  we have
$$ K=\frac{1+i+2j+3k}{-1}=-QF_1$$
where $QF_1$ is $1st$ Fibonacci quaternion. And then using the formula $$\frac{1}{\alpha-\beta}\big( \frac{B\underline{\beta}\beta^{n+1}}{1-\beta}-\frac{A\underline{\alpha}\alpha^{n+1}}{1-\alpha} \big)$$ we get
$$\frac{1}{\sqrt{5}}\big( \underline{\alpha} \alpha^{n+2} - \underline{\beta} \beta^{n+2} \big) =QF_{n+2}$$
where $QF_{n+2}$ is $(n+2)th$ Fibonacci quaternion. Thus, we have this:
$$ \sum_{k=0}^n{Q_{w,k}}=\frac{1}{\alpha-\beta}\big( \frac{B\underline{\beta}\beta^{n+1}}{1-\beta}-\frac{A\underline{\alpha}\alpha^{n+1}}{1-\alpha} \big) + K=QF_{n+2}-QF_1.$$
The last equation is a summation formula for Fibonacci quaternion and this formula is given by Halici in [5].

Now if we consider the Pell quaternions that is when $ a=0, b=p=2, q=1, \alpha-\beta =2\sqrt{2}$, then we obtain that
$$K=\frac{1+i+3j+7k}{-2}=\frac{-1}{2}QPL_0, $$
where $QPL_0$ is $0th$ modified Pell quaternion. It should be noted that, in [3], Bolat and Ipek introduced this quaternion type as Pell-Lucas quaternion, but because of the initial values it must be the modified Pell quaternion (see, [2]).\\\\
And using by the formula $\frac{1}{\alpha-\beta}\big( \frac{B\underline{\beta}\beta^{n+1}}{1-\beta}-\frac{A\underline{\alpha}\alpha^{n+1}}{1-\alpha} \big)$ one can obtain this;\\\\
$$\frac{1}{\alpha-\beta}\big( \frac{B\underline{\beta}\beta^{n+1}}{1-\beta}-\frac{A\underline{\alpha}\alpha^{n+1}}{1-\alpha} \big)=\frac{1}{2}\big( \frac{ \underline{\alpha} \alpha^{n+1} + \underline{\beta} \beta^{n+1}}{2} \big)=\frac{1}{2}QPL_{n+1},$$\\\\
where $QPL_{n+1}$ is $(n+1)th$ modified Pell quaternion. As a result we obtain that
$$\sum_{k=0}^n{Q_{w,k}}=\frac{1}{\alpha-\beta}\big( \frac{B\underline{\beta}\beta^{n+1}}{1-\beta}-\frac{A\underline{\alpha}\alpha^{n+1}}{1-\alpha} \big) + K=\frac{1}{2}(QPL_{n}-QPL_0).$$
So, this result is the same as the result which is given by Bolat and Ipek in [3].\\\\
Now, if we write the values $a=0, b=1, p=1, q=2$ in $(3.36)$ then we have
$$K=\frac{(1+i+3j+7k)}{-2}=\frac{-1}{2}QJ_1$$
where $QJ_1$ is $1st$ Jacobsthal quaternion.\\\\
From the formula $$\frac{1}{\alpha-\beta}\big( \frac{B\underline{\beta}\beta^{n+1}}{1-\beta}-\frac{A\underline{\alpha}\alpha^{n+1}}{1-\alpha} \big)$$ we write
$$\frac{1}{\alpha-\beta}\big( \frac{B\underline{\beta}\beta^{n+1}}{1-\beta}-\frac{A\underline{\alpha}\alpha^{n+1}}{1-\alpha} \big)=\frac{1}{3}(\underline{\alpha} 2^{n+2}+\underline{\beta} (-1)^{n+1})=QJ_{n+2}$$\\\\
So, we obtain the following formula.
$$\sum_{k=0}^n{QJ_{w,k}}=QJ_{n+2}-\frac{1}{2}QJ_1$$
Likewise using the equation $(3.36)$ one can easily prove that the other summation formulas. So, we conclude that the formula $(3.36)$ is the general case of the other sum formulas.\end{proof}
\begin{thm}
The norm value of Horadam quaternions $Q_{w,n}$ is follows.
\begin{equation}
Nr^2(Q_{w,n})=\frac{1}{p^2+4q}(b^2A+2abqB+a^2q^2C)
\end{equation}
where $A, B, C$ are follows.
$$A=\alpha^{2n} (1+\alpha^2+\alpha^4+\alpha^6)+\beta^{2n}(1+\beta^2+\beta^4+\beta^6)-2(-q)^n(1-q+q^2-q^3),$$
$$B=\alpha^{2n-1} (1+\alpha^2+\alpha^4+\alpha^6)+\beta^{2n-1}(1+\beta^2+\beta^4+\beta^6)-p(-q)^{n-1}(1+q+q^2+q^3),$$
$$C=\alpha^{2n} (1+\alpha^2+\alpha^4+\alpha^{-2})+\beta^{2n}(1+\beta^2+\beta^4+\beta^{-2})-2(-q)^n(1-q+q^2-(-q)^{-1}).$$
\end{thm}
\begin{proof}
Using the definitions norm and Binet formula $W_n=bT_n+aqT_{n-1}$ one can write
$$Nr^2(Q_{w,n})=W^2_n+W^2_{n+1}+W^2_{n+2}+W^2_{n+3}$$
\scriptsize {$$Nr^2(Q_{w,n})=(bT_n+aqT_{n-1})^2+(bT_{n+1}+aqT_{n})^2+(bT_{n+2}+aqT_{n+1})^2+(bT_{n+3}+aqT_{n+2})^2.$$ }
\normalsize Also, using the values $\alpha+\beta=p, \alpha-\beta=\sqrt{p^2+4q}, \alpha\beta=-q$ and
$$ T_n= \frac{1}{\sqrt{p^2+4q}} \big\{ \big(\frac{p+\sqrt{p^2+4q}}{2} \big)^n - \big(\frac{p-\sqrt{p^2+4q}}{2} \big)^n \big\}$$
we can conclude that
$$Nr^2(Q_{w,n})=\frac{1}{p^2+4q}(b^2A+2abqB+a^2q^2C)$$ which is desired result.
\end{proof}
\begin{rem}
Some special cases of Theorem $(3.9)$ can be listed as follows.\\\\
i) Writing the values $a=0, b=p=q=1, \alpha - \beta =\sqrt{5}$ in the equation $(3.37)$ we find
$$Nr^2(QF_n)=\frac{1}{5} (\alpha^{2n} (15+6\sqrt{5})+\beta^{2n}(15-6\sqrt{5})),$$
which is the norm value of Fibonacci quaternions.\\\\
ii) If we write the values $a=0, b=q=1, p=2, \alpha=1+\sqrt{2}, \beta=1-\sqrt{2}$ in the equation $(3.37)$, then we get
$$Nr^2(QP_n)=\frac{1}{8} (\alpha^{2n} (120+84\sqrt{2})+\beta^{2n}(120-84\sqrt{2}))$$
which is the norm value for Pell quaternions and it is given by Anetta and Iwona in [16].\\\\
iii) If we write the values $a=0, b=1, p=1, q=2, \alpha-\beta =3$ in the equation $(3.37)$, then we have
$$Nr^2(QJ_n)=\frac{1}{9} (85(2)^{2n} +10(-1)^n (2)^{n} +4)$$
The last equation gives the norm value for Jacobsthal quaternions and this value is calculated by Anetta and Wloch in [17].\\\\
In the same way using the equation $(3.37)$ one can easily prove the other norm values.
\end{rem}
\maketitle
\section{\textbf{Conclusion}}
In this paper, firstly we have defined the sequence of  Horadam quaternions that generalize all the other quaternions  by a recurrence relation of second order. And then we have presented some properties involving these sequence. Moreover, for these quaternions we give the Binet formula, generating function, Cassini identity, norm value and sum formula. In the future, we intend to introduce  Horadam octonions and give fundamental properties for octonions of this type.

\end{document}